\documentclass[preprint,11pt]{elsarticle}

\usepackage{amssymb}
\usepackage{amsmath}
\usepackage{amsthm}
\usepackage{cases}
\usepackage{amsfonts}
\usepackage{graphicx}
\usepackage{colortbl,dcolumn}
\usepackage{float}
\usepackage{paralist}  % generate "compactenum"

\newtheorem{lemma}{\bf Lemma}[section]
\newtheorem{theorem}{\bf Theorem}[section]
\newtheorem{proposition}{\bf Proposition}[section]
\newtheorem{corollary}{\bf Corollary}[section]
\newtheorem{definition}{\bf Definition}[section]
\newtheorem{remark}{\bf Remark}[section]

\newtheorem*{assumption}{\bf Assumption}

\numberwithin{equation}{section}%对公式以节{section}为基础进行编号.变成（1.1）

\makeatletter
\def\ps@pprintTitle{%
	\let\@oddhead\@empty
	\let\@evenhead\@empty
	\let\@oddfoot\@empty
	\let\@evenfoot\@oddfoot
}
\makeatother

\begin{document}

\begin{frontmatter}
 \pdfoutput=1


\title{Asymptotic behavior for the fast diffusion equation with absorption and singularity}


\author[label1,label3]{Changping Xie}
\ead{xiechangping314@163.com}

\author[label1]{Shaomei Fang}
\ead{dz90@scau.edu.cn}

\author[label2,label3]{Ming Mei\corref{cor1}}
\ead{ming.mei@mcgill.ca}
\cortext[cor1]{Corresponding author}

\author[label4]{Yuming Qin\corref{cor1}}
\ead{yuming@dhu.edu.cn}

\affiliation[label1]{organization={College of Mathematics and Informatics},%Department and Organization
            addressline={South China Agricultural University},
            city={Guangzhou},
            postcode={510642},
            state={Guangdong},
            country={P.R.China}}

\affiliation[label2]{organization={Department of Mathematics},%Department and Organization
	addressline={Champlain College},
	city={Saint-Lambert},
	postcode={J4P 3P2},
	state={Quebec},
	country={Canada}}

\affiliation[label3]{organization={Department of Mathematics and Statistics},%Department and Organization
	addressline={McGill University},
	city={Montreal},
	postcode={H3A 2K6},
	state={Quebec},
	country={Canada}}

\affiliation[label4]{organization={Institute of Nonlinear Science},%Department and Organization
	addressline={Donghua University},
	city={Shanghai},
	postcode={200051},
	%state={},
	country={P.R.China}}
\begin{abstract}
This paper is concerned with the weak solution for the fast diffusion equation with absorption and singularity in the form of $u_t=\triangle u^m -u^p$.  We first prove the   existence and decay estimate of weak solution when the fast diffusion index satisfies $0<m<1$ and the absorption index is $p>1$. Then we show the asymptotic convergence of weak solution to the corresponding  Barenblatt solution for $\frac{n-1}{n}<m<1$ and $p>m+\frac{2}{n}$ via the entropy dissipation method combining the generalized Shannon's inequality and  Csisz$\mathrm{\acute{a}}$r-Kullback inequality. The singularity of spatial diffusion causes us the technical challenges for the asymptotic behavior of weak solution.
\end{abstract}




\end{frontmatter}


\section{Introduction}\label{}
We consider the following fast diffusion equation with absorption and singularity:
\begin{equation}\label{fast diffusion equation with absorption}
	u_t=\Delta u^{m}-u^p,\quad x\in \mathbb{R}^n,~t>0,
\end{equation}
with initial data
\begin{equation}\label{initial data}
		u(x,0)=u_0(x),\quad x\in \mathbb{R}^n,
\end{equation}
where $n\geq2$, $p>1$, $0<m<1$ and the initial data $u_0(x)$ is a non-negative function in $L^1(\mathbb{R}^n)\cap L^\infty(\mathbb{R}^n)$. The unknown function $u=u(x,t)$ is the density of particles. The absorption term $u^p$ can cause the mass conservation law fail to be true, more precisely, the mass is monotonically decreasing. $\Delta u^{m}$ with $0<m<1$ is the singular diffusion with more physical sense, that is, the diffusive velocity of particles depends on the current particles density, and the smaller density, the larger diffusion, in particular, when $u=0$, the diffusion speed is $\infty$, which is the so-called fast diffusion.

{\bf Background of study}. The fast diffusion equation
\begin{equation}\label{fast diffusion equation}
	u_t=\Delta u^{m},\quad x\in \mathbb{R}^n,~t>0,
\end{equation}
with $0<m<1$, which is an important model for singular nonlinear diffusive phenomena, can be used to model gas-kinetics, thin liquid film dynamics and plasma in nuclear reactors \cite{Application fast 1,Application fast 3,Application fast 2}. Since the 70s the equation \eqref{fast diffusion equation} has been thoroughly studied and a series of celebrated results have been obtained, see  \cite{fast 1,fast 2,fast 3,fast 4,fast 5,fast 6} and the references therein. On the other hand, when $m>1$, the equation \eqref{fast diffusion equation} is the slow diffusion equation, the so-called degenerate diffusion equation, which has been also extensively studied \cite{Continue,Application fast 2,pme 1,pme 2}.

 When the fast diffusion equations are affected by the external sources, the mass of flows is no long conservative. There are two different external sources. One is with external supplement, which is called the expansion phenomenon and presented in the form of
 \begin{equation}\label{fast diffusion equation with expansion}
	u_t=\Delta u^{m}+u^p,\quad x\in \mathbb{R}^n,~t>0.
\end{equation}
The other is with the external omission, the so-called absorption phenomenon, which is described by  \eqref{fast diffusion equation with absorption}, namely,
\[
u_t=\Delta u^{m}-u^p,\quad x\in \mathbb{R}^n,~t>0.
\]

For the fast diffusion equation with expansion \eqref{fast diffusion equation with expansion} subjected to the initial data \eqref{initial data}, the mass is monotonically increasing. In 1991, Anderson \cite{expension 1,expension 2} first studied the
existence and uniqueness of weak solution for \eqref{fast diffusion equation with expansion} and \eqref{initial data}. Then,
 Galaktionov and V$\mathrm{\acute{a}}$zquez \cite{expension 6} studied the ``peak-like" incomplete blow-up solution, and showed the possible continuation of incomplete blow-up solutions to \eqref{fast diffusion equation with expansion}
 after the blow-up time, and further discussed the different continuation modes in terms of the exponents $m$ and $p$. But when the same technique of analysis was
 applied to \eqref{fast diffusion equation with absorption}, they recognized that no such a continuation of incomplete blow-up solution exists if $p+m\leq0$, namely, 
 the  complete  extinction solution, and there exists a nontrivial continuation if $p+m>0$.
%In a bounded domain of $\mathbb{R}^n$, the solution blowed up in finite time for sufficiently large initial data \cite{expension 3,expension 4} and vanished in finite time for sufficiently small initial data if $p>m$ \cite{expension 5}.

For the fast diffusion equation with absorption \eqref{fast diffusion equation with absorption} subjected to the initial data \eqref{initial data}, the mass is monotonically decreasing and the limiting mass is strictly positive (see Lemma \ref{mass} later). In 1991, Peletier \cite{Existence 1} first investigated the  existence and uniqueness of solution to problem \eqref{fast diffusion equation with absorption}-\eqref{initial data}, which was then significantly developed by Borelli and Ughi \cite{Existence 2} and Leoni \cite{Singular solution 1}, respectively. For asymptotic behavior, the critical exponent $p=m+\frac{2}{n}$ is crucial, as the competition between the diffusion $-\Delta u^{m}$ and the absorption $u^{p}$ in \eqref{fast diffusion equation with absorption}. More precisely, Peletier \cite{Asymptotic 2} proved the asymptotic results uniformly on moving boundary set of $\{x\in \mathbb{R}^n:\vert x\vert<at^{\frac{1}{b}}\}$ for $p>1$, $p>m+\frac{2}{n}$ and $m<p<m+\frac{2}{n}$, where $a$ and $b$ are some numbers specified in \cite{Asymptotic 2}. Then, Ferreira and V$\mathrm{\acute{a}}$zquez \cite{Extinction 1} showed the finite-time extinction for problem \eqref{fast diffusion equation with absorption}-\eqref{initial data} in one-dimension, which was the main qualitative feature of the bounded solutions in the range $0<p<1$. Furthermore, Benachour et al. \cite{Asymptotic 1} gave the asymptotic behavior related to  Barenblatt profile with a logarithmic scaling in the critical exponent $p=m+\frac{2}{n}$.

However, as far as we know, for the fast diffusion equation with absorption \eqref{fast diffusion equation with absorption}-\eqref{initial data}, the study of the asymptotic convergence of weak solution to the Barenblatt profile in full space is almost never related, due to some technical issues. To attach this problem is the main purpose of this paper.
The adopted approach is the technical entropy dissipation method.

The entropy dissipation method, which was developed by Toscani and Carrillo \cite{Begin 2,Begin 1} to investigate the asymptotic behavior of the solution to diffusion equations and Fokker-Planck-type equations, was applied to prove the asymptotic behavior of the solution to the degenerate diffusion equation ($m>1$)  and fast diffusion equation ($\frac{n-2}{n}<m<1$) by Carrillo and Toscani \cite{Apply 1} and by Carrillo and V\'{a}zquez \cite{Apply 2}, respectively. They \cite{Apply 1,Apply 2} showed the convergence of the weak solution to Barenblatt profiles with algebraic rates in the corresponding exponent interval. For $\frac{n-2}{n}<m<1$, the Barenblatt solution to \eqref{initial data}-\eqref{fast diffusion equation} is given by the formula \cite{Apply 2}
\begin{equation}\label{Barenblatt solution u}
	U_M(x,t)=\left(\frac{\alpha t}{\vert x \vert^2+\beta t^{\frac{2}{\lambda}}}\right)^{\frac{1}{1-m}},
\end{equation}
where $\lambda=2-n(1-m)$, $\alpha=2m\lambda/(1-m)$ and $\beta>0$ is a constant such that $\Vert U_M \Vert_1=M$.
The stationary solution to the nonlinear Fokker-Planck equation
\begin{equation}\label{Fokker Planck equation}
	\rho_s=\text{div}_{y}(\rho y+\nabla_{y}\rho^{m}),\quad y\in \mathbb{R}^n,~s>0,
\end{equation}
can be directly obtained by rescaling the time shifted profiles
\begin{equation}\label{Barenblatt solution rho}
	U_M(x,t+\lambda^{-1})\to\rho_M(y)=\left(\frac{\gamma}{\vert y \vert^2+\theta}\right)^{\frac{1}{1-m}},
\end{equation}
where $\gamma=\alpha/\lambda=2m/(1-m)$, $\theta=\beta\lambda^{-2/\lambda}$. More research on the entropy dissipation method can be found in \cite{Apply 4,Apply 3,Apply 5,Apply 6} and the literatures therein.

{\bf Difficulty and strategy}.  Regarding the fast diffusion problem with absorption \eqref{fast diffusion equation with absorption}-\eqref{initial data}, in order to investigate the asymptotic behavior of the weak solution, we first need to clarify what will be the targeted asymptotic profile. Note that, in the  absorbing fast diffusion case,
 the total mass of flow is monotonically decreasing and the limiting mass is strictly positive (see Lemma \ref{mass} later), therefore the expected asymptotic profile for the weak solution is not the standard Barenblatt solution as showed in \eqref{Barenblatt solution u}. Let us define the total mass of $u(x,t)$ by $\widetilde{M}(t):=\int_{\mathbb{R}^n} u(x,t) dx$,  we are reasonable to construct a special Barenblatt solution to the standard fast diffusion equation \eqref{fast diffusion equation} related to the total mass $\widetilde{M}(t)$ of \eqref{fast diffusion equation with absorption} by
\begin{equation}\label{Barenblatt solution u M}
	U_{\widetilde{M}(t)}(x,t)=\left(\frac{\alpha t}{\vert x \vert^2+\beta_{\widetilde{M}(t)} t^{\frac{2}{\lambda}}}\right)^{\frac{1}{1-m}},
\end{equation}
where $\beta_{\widetilde{M}(t)}>0$ is a function such that $\Vert U_{\widetilde{M}(t)} \Vert_1=\widetilde{M}(t)$. We expect this modified Barenbalatt solution $U_{\widetilde{M}(t)}(x,t)$ as the asymptotic profile to the original fast diffusion equation with absorption \eqref{fast diffusion equation with absorption}-\eqref{initial data}. The corresponding stationary solution to \eqref{Fokker Planck equation} can be derived as
\begin{equation}\label{Barenblatt solution rho M}
	U_{\widetilde{M}(t)}(x,t+\lambda^{-1})\to\rho_{M(s)}(y)=\left(\frac{\gamma}{\vert y \vert^2+\theta_{M(s)}}\right)^{\frac{1}{1-m}},
\end{equation}
where $\theta_{M(s)}=\beta_{\widetilde{M}(t)}\lambda^{-2/\lambda}>0$ is a function such that $\Vert \rho_{M(s)}\Vert_{1}=M(s)=\widetilde{M}(t)$.

In order to prove the convergence of the weak solution to the specially constructed Barenblatt profile \eqref{Barenblatt solution u M}, we adopt the entropy dissipation method as mentioned above. After carrying out the suitable variable scalings (see \eqref{rescaled variables} later), we reduce the original IVP \eqref{fast diffusion equation with absorption}-\eqref{initial data}  to the new system presented in \eqref{rho equation}-\eqref{rho initial data}. Then we show that the rescaled solution $\rho(y,s)$  eventually converges to the profile $\rho_{M_\infty}(y)$, by first showing that it converges to the time dependent profile $\rho_{M(s)}(y)$.

Here are some technical issues in the proof. It is worth noting that Carrillo and Fellner \cite{Apply 4} studied the long-time asymptotic of the solution to the degenerate (slow) diffusion equation \eqref{fast diffusion equation with absorption}-\eqref{initial data} with $m>1$.  Since the diffusion coefficient of \eqref{fast diffusion equation with absorption} is $mu^{m-1}$, the larger $u$, the larger
diffusion power for $m>1$, which can easily derive the $L^m$ boundedness of $u$. The $L^m$-bound of the solution plays an important role in showing the asymptotic convergence. However, in the fast diffusion case of $m<1$, we come across the reverse situation. When $m<1$, the diffusion coefficient $mu^{m-1}\to\infty$ as $u\to0$, that is, $u^{m-1}$ is unbounded, and $u\notin L^m$ in general, which requires a number of changes in the arguments. In order to overcome these obstacles, we use the generalized Shannon's inequality to prove the boundedness of the second moment, and then prove  $L^m$ boundedness of $u$ by the $L^1$ decay estimate and the bound of the second moment of $u$. After a series of meticulous estimates, we can prove the asymptotic convergence of weak solution to the corresponding  Barenblatt solution.

{\bf Main results}. Before stating our main results,  we first impose the following assumption:

\begin{assumption} Let $n\geq2$. The initial data $u_{0}$ is a non-negative function satisfying
\begin{equation*}
	u_{0}\in L^1(\mathbb{R}^n)\cap L^\infty(\mathbb{R}^n)\; \mathrm{with}\;u_0^m\in H^1(\mathbb{R}^n).
\end{equation*}	
\end{assumption}

Although the existence of the defined weak solution has been proved in \cite{Existence 1}, in order to study the asymptotic property of \eqref{fast diffusion equation with absorption}-\eqref{initial data}, we need higher regularities for the weak solution, and then we give another definition of the weak solution.

\begin{definition}[Weak solution]\label{weak solution} Let the Assumption hold and $0<m<1$. For $T>0$ and $Q_T=\mathbb{R}^n\times [0,T)$, $u(x,t)$ is called a weak solution to \eqref{fast diffusion equation with absorption}-\eqref{initial data} if \\
{\rm(i)}\,$u(x,t)\geq0$ for almost all $(x,t)\in Q_T$.\\
{\rm(ii)}\,$u\in L^\infty(0,T;L^{2}(\mathbb{R}^n)\cap L^{\infty}(\mathbb{R}^n))$ and $\nabla u^m\in L^\infty(0,T;L^{2}(\mathbb{R}^n)$.\\
{\rm(iii)}\,For any $\phi\in C^{\infty}(Q_T)$ with $\mathrm{supp}\;\phi(\cdot,t)\subset\subset\mathbb{R}^n$ and $0<t<T$,
\begin{equation}\label{local existence equation}
	\begin{aligned}
      &\int u(t)\phi(t)dx-\int u_0 \phi(0)dx\\
      &=\int\limits_{0}^{t}\int u(\tau)\partial_{\tau}\phi(\tau)dxd\tau-\int\limits_{0}^{t}\int \left(\nabla u(\tau)^m\cdot\nabla\phi(\tau)+u(\tau)^{p}\phi(\tau)\right)dxd\tau.
    \end{aligned}
\end{equation}	
\end{definition}

Before we give the asymptotic result of the weak solution to \eqref{fast diffusion equation with absorption}-\eqref{initial data}, we show the time-local existence and decay estimate of weak solutions to \eqref{fast diffusion equation with absorption}-\eqref{initial data}.

\begin{theorem}[Local existence] \label{local existence} Let the Assumption hold. Suppose $0<m<1$ and $p>1$. Then there exist a weak solution $u(x,t)$ on $[0,T)$ to \eqref{fast diffusion equation with absorption}-\eqref{initial data} in the sense of Definition \ref{weak solution}.
\end{theorem}

For every $T<\infty$, if $u(x,t)$ is a weak solution to \eqref{fast diffusion equation with absorption}-\eqref{initial data} in $Q_T$, we say that $u(x,t)$ is global.
Then we present the decay properties for the weak solution to \eqref{fast diffusion equation with absorption}-\eqref{initial data}.

\begin{theorem}[Decay estimate]\label{decay estimate}
	Let the Assumption hold. Suppose $0<m<1$ and $p>1$. Then for $1\leq r \leq \infty$, problem \eqref{fast diffusion equation with absorption}-\eqref{initial data} has a  global weak solution $u(x,t)$ satisfying
	\begin{equation}\label{r norm estimation of u}
		\Vert u(t) \Vert_{r}\leq \Vert u_0 \Vert_{r}, \quad t\geq0,
	\end{equation}	
	and
	\begin{equation}\label{t dacay estimation of u 1}
		\Vert u(t) \Vert_{r}\leq C(\Vert u_0 \Vert_{1},\Vert u_0 \Vert_{\infty},r,p)(1+t)^{-\frac{r-1}{r(p-1)}}, \quad t\geq0.
	\end{equation}
	For the case of $r\geq3-m$, we have
	\begin{equation}\label{t dacay estimation of u 2}
		\Vert u(t) \Vert_{r}\leq C(\Vert u_0 \Vert_{1},\Vert u_0 \Vert_{\infty},n,r,p,m)(1+t)^{-\max{\{\frac{r-1}{r(p-1)},\frac{n(r-1)}{2r(2-m)}}\}}, \quad t\geq0.
	\end{equation}
\end{theorem}

At last, we show the asymptotic behavior of weak solution to \eqref{fast diffusion equation with absorption}-\eqref{initial data}.

\begin{theorem}[Convergence to Barenblatt Solution] \label{asymptotic behavior}Let the Assumption hold. Suppose $\frac{n-1}{n}<m<1$, $p>m+\frac{2}{n}$ and $\int u_0\vert x\vert^2dx<\infty$. Then the weak solution $u(x,t)$ to \eqref{fast diffusion equation with absorption}-\eqref{initial data} in the sense of Definition \ref{weak solution} converges to the corresponding Barenblatt solution in the form:
\begin{equation}\label{asymptotic behavior u 1}
\Vert u-U_{\widetilde{M}(t)}\Vert_{1}\leq C(1+\lambda t)^{-\frac{1}{\lambda}\min{\{1,\delta\}}},\quad \mathrm{for\; } \delta\neq1,
\end{equation}	
and
\begin{equation}\label{asymptotic behavior u 2}
	\Vert u-U_{\widetilde{M}(t)}\Vert_{1}\leq C(1+\lambda t)^{-\frac{1}{\lambda}}\ln{(1+\lambda t)},\quad \mathrm{for\; } \delta=1,
\end{equation}
where $\delta=np-nm-2$, $U_{\widetilde{M}(t)}$ is defined as \eqref{Barenblatt solution u M} and the constant $C>0$ is depending on $\Vert u_0 \Vert_1$, $\Vert u_0 \Vert_\infty$, $n$, $p$ and $m$. Moreover, for $1<r<\infty$, we have
\begin{equation}\label{asymptotic behavior u 3}
		\Vert u-U_{\widetilde{M}(t)}\Vert_{r}\leq C (1+\lambda t)^{-\frac{n(r-1)+\min{\{1,\delta\}}}{\lambda r}},\quad \mathrm{for\; } \delta\neq1,
\end{equation}
and
\begin{equation}\label{asymptotic behavior u 4}
		\Vert u-U_{\widetilde{M}(t)}\Vert_{r}\leq C (1+\lambda t)^{-\frac{n(r-1)+1}{\lambda r}}\ln{(1+\lambda t)},\quad \mathrm{for\; } \delta=1,
\end{equation}
where the constant $C>0$ is depending on $\Vert u_0 \Vert_1$, $\Vert u_0 \Vert_\infty$, $n$, $p$, $m$ and $r$.
\end{theorem}

This paper is organized as follows. In Section 2, we prepare some notations and  lemmas which will be used often in the other sections. In Section 3, we prove Theorem \ref{local existence} and Theorem \ref{decay estimate}, namely the existence and decay estimate for the weak solution to \eqref{fast diffusion equation with absorption}-\eqref{initial data}. Then, we prove Theorem \ref{asymptotic behavior} in Section 4, that is, the asymptotic behavior of  the weak solution to \eqref{fast diffusion equation with absorption}-\eqref{initial data}.

\section{Preliminaries}

In this section, we show some notations and  lemmas which will be used often in the other sections.

For simplicity, we denote $\Vert\cdot\Vert_{L^{r}(\mathbb{R}^n)}$ by $\Vert\cdot\Vert_r$, $\int_{\mathbb{R}^n}\cdot dx$  by $\int\cdot dx$, and $L_{2}^{1}(\mathbb{R}^n)\equiv\{f\in L^1(\mathbb{R}^n);\vert x\vert^{2}f\in L^1(\mathbb{R}^n)\}$. $C$ represents the constant which may be different from line by line and $C=C(\cdot,\cdots,\cdot)$ denotes the constant that depends only on the variables in parentheses.

\begin{lemma}[\cite{Contraction of the difference of exponential function}] \label{Contraction of the difference of exponential function}
	Let $a$ and $b$ be arbitrary positive numbers. Then we have
\begin{equation}
	\vert a^r-b^r\vert\leq
	\begin{cases}
		2^{r-1}r(a^{r-1}+b^{r-1})\vert a-b\vert,\quad & r>1,\\
		\vert a-b\vert^r,\quad & 0<r\leq1.
	\end{cases}
\end{equation}	
\end{lemma}

\begin{lemma}[Nash inequality \cite{Nash inequality 1,Apply 3}] \label{Nash inequality-1}
	There exists a constant  $C>0$ such that for all $f\in L^{1}(\mathbb{R}^n)\cap H^{1}(\mathbb{R}^n)$,
	\begin{equation}\label{Nash inequality}
		\Vert f\Vert_{2}^{1+\frac{2}{n}}\leq C\Vert f\Vert_{1}^{\frac{2}{n}}\Vert \nabla f\Vert_{2},
	\end{equation}	
where the constant $C$ depend only on $n$.
\end{lemma}

\begin{lemma}[Cut-off function \cite{cut-off function}] \label{cut-off function-1}
	Let  $R>0$ and $\eta(x)=\eta(\vert x \vert)$ be defined as
	\begin{equation}\label{cut-off function}
		\eta(x)=
		\begin{cases}
			1,&\quad 0\leq \vert x \vert<R,\\
			\exp\left(1-\frac{R}{2R-\vert x \vert}\right),&\quad R\leq \vert x \vert<2R,\\
			0,&\quad \vert x \vert\geq 2R.
		\end{cases}
	\end{equation}	
Then, it holds that
\begin{align}
	&\vert \nabla\eta(x) \vert\leq \frac{C}{a^2R}\eta(x)^{1-a},&\label{cut off function estimate 1}\\
	&\vert \Delta\eta(x) \vert\leq \frac{C}{a^4R^2}\eta(x)^{1-a},&\label{cut off function estimate 2}
\end{align}
for all $x\in\mathbb{R}^n$ and all $0<a<1$, where $C$ is is a constant depending only on $n$.
\end{lemma}

\begin{lemma}[Generalized Shannon's inequality \cite{Generalized Shannon's inequality}] \label{Generalized Shannon's inequality}
	Let  $n\geq2$, and $\frac{n}{n+2}<m<1$. Then there exists a constant $C=C(n)>0$ such that for any $f\in L_{2}^{1}(\mathbb{R}^n)$, we have
	\begin{equation}\label{Shannon's inequality}
		\int\vert f \vert^{m}dx\leq C\Vert f \Vert^{m(1-\sigma)}_{1} \left(\int\vert x \vert^{2}\vert f(x) \vert dx\right)^{m\sigma},
	\end{equation}	
	where $\sigma=\frac{n(1-m)}{2m}$.
\end{lemma}

\begin{lemma}[Csisz$\mathrm{\acute{a}}$r-Kullback inequality \cite{Apply 3,CK inequality 1}] \label{CK inequality 1}
	Let  $0<m<1$  and a nonnegative function $f\in L^{1}(\mathbb{R}^n)$ with $\Vert f \Vert_1=M$. Suppose that a nonnegative function $g\in L^{1}(\mathbb{R}^n)\cap L^{2-m}(\mathbb{R}^n)$ satisfies $\Vert g \Vert_1=M$ and
	\begin{equation*}
		c_{m,g}=M^{\frac{m-2}{2}}\left(\frac{2}{m}\int g(x)^{2-m}dx\right)^{\frac{1}{2}},
	\end{equation*}
then it holds that	
	\begin{equation}\label{CK inequality 2}
		\Vert f-g \Vert_{1}\leq c_{m,g}\sqrt{D(f|g)},
	\end{equation}	
	where
	\begin{eqnarray*}
		D(f|g)&=&\frac{1}{m-1}\int f(x)^{m}dx-\frac{1}{m-1}\int g(x)^{m}dx \\
              & &-\frac{m}{m-1}\int g(x)^{m-1}(f(x)-g(x))dx.
	\end{eqnarray*}
\end{lemma}

\section{Existence and decay estimate}
As the equation \eqref{fast diffusion equation with absorption} is a quasi-linear parabolic equation of singular type, we consider the following approximating problem of \eqref{fast diffusion equation with absorption}-\eqref{initial data} to justify all the formal arguments:
\begin{align}
	&\partial_{t}u_{\varepsilon}=\Delta \left(u_{\varepsilon}^{m}+\varepsilon u_{\varepsilon}\right)-u_{\varepsilon}^p,\quad && x\in \mathbb{R}^n,~t\in(0,T),&\label{approximating equation u}\\
	&u_{\varepsilon}(x,0)=(u_{0}*\widetilde{\eta}_{\varepsilon})(x)=u_{0\varepsilon}(x),\quad && x\in \mathbb{R}^n,&\label{approximating initial data u}
\end{align}
where $\varepsilon>0$ and non-negative mollifier $\widetilde{\eta}_{\varepsilon}\in C_{0}^{\infty}(\mathbb{R}^n)$ satisfies $\widetilde{\eta}_{\varepsilon}\subset\overline{B_{\varepsilon}(0)}$ and $\Vert\widetilde{\eta}_{\varepsilon}\Vert_{1}=1$. By the similar argument as in \cite{Existence of solution for approximating equation, Equicontinuous of solution for approximating equation}, we know that \eqref{approximating equation u}-\eqref{approximating initial data u} has a unique classical solution $u_{\varepsilon}$, and
\begin{equation}\label{positivity and boundedness of u var}
	0\leq u_{\varepsilon}\leq C,
\end{equation}
where $C$ is a positive constant which does not depend on $\varepsilon$. Furthermore, there exists a subsequence, which we still define as $\{u_{\varepsilon}\}$, such that
\begin{equation}\label{continuous u}
	u_{\varepsilon}\to u,\quad \mathrm{in\;}C(K)
\end{equation}
for every compact set $K\subset (0,T;\mathbb{R}^n)$. In what follows, we will begin the proof of Theorem \ref{local existence}.
\vspace{1.5ex}\\
\textbf{Proof of Theorem \ref{local existence}.}
Multiplying the equation \eqref{approximating equation u} by $\partial_t \left(u_{\varepsilon}^{m}+\varepsilon u_{\varepsilon}\right)$ and integrating with respect to $x$ and $t$, we get
%% 先在(0,t)上积分，再取sup_(0,t)
\begin{eqnarray}\label{approximating estimate 1}		
& &\frac{4m}{(m+1)^2}\int\limits_{0}^{T}\int\left\vert\partial_{t}u_{\varepsilon}^{\frac{m+1}{2}}\right\vert^2dxdt
+\varepsilon\int\limits_{0}^{T}\int\vert\partial_{t}u_{\varepsilon}\vert^2dxdt \notag \\
& & \ \ \ +\frac{1}{2}\sup_{0<t<T}\int\vert\nabla\left(u_{\varepsilon}^{m}+\varepsilon u_{\varepsilon}\right)\vert^2dx \notag \\
& &\ \ \ +\frac{m}{p+m}\sup_{0<t<T}\int u_{\varepsilon}^{p+m}dx+\frac{\varepsilon}{p+1}\sup_{0<t<T}\int u_{\varepsilon}^{p+1}dx \notag \\
&&=\frac{1}{2}\int\vert\nabla\left(u_{0\varepsilon}^{m}+\varepsilon u_{0\varepsilon}\right)\vert^2dx+\frac{m}{p+m}\int u_{0\varepsilon}^{p+m}dx+\frac{\varepsilon}{p+1}\int u_{0\varepsilon}^{p+1}dx,
\end{eqnarray}
which combined with \eqref{positivity and boundedness of u var} can yield
\begin{eqnarray}\label{approximating estimate 3}
&&\sup_{0<t<T}\int\left\vert\nabla u_{\varepsilon}^{\frac{m+1}{2}}\right\vert^2dx \notag \\
&&\leq\frac{(m+1)^2}{4m^2}\Vert u_{\varepsilon} \Vert_{\infty}^{1-m}\sup_{0<t<T}\int\vert\nabla\left(u_{\varepsilon}^{m}+\varepsilon u_{\varepsilon}\right)\vert^2dx \notag \\
&&\leq C(\Vert u_{0}\Vert_{1},\Vert u_{0}\Vert_{\infty},\Vert\nabla u_{0}^m\Vert_{2},p,m).
\end{eqnarray}

On the other hand, we multiply the equation \eqref{approximating equation u} by $u_{\varepsilon}^{m}$ and integrate with respect to $x$ and $t$, and then we obtain
\begin{eqnarray}\label{approximating estimate 2}
&&\frac{1}{m+1}\sup_{0<t<T}\int u_{\varepsilon}^{m+1}dx+\frac{4m\varepsilon}{(m+1)^2}\int\limits_{0}^{T}\int\left\vert\nabla u_{\varepsilon}^{\frac{m+1}{2}}\right\vert^2dxdt
\notag\\
& & \ \ \ +\int\limits_{0}^{T}\int\vert\nabla u_{\varepsilon}^m\vert^2dxdt
+\int\limits_{0}^{T}\int u_{\varepsilon}^{p+m}dxdt \notag \\
&&=\frac{1}{m+1}\int u_{0\varepsilon}^{m+1}dx.
\end{eqnarray}
From  \eqref{approximating estimate 1}, \eqref{approximating estimate 3} and \eqref{approximating estimate 2}, we have $u_{\varepsilon}^{\frac{m+1}{2}}\in L^\infty(0,T;H^{1}(\mathbb{R}^n))\cap H^1(0,T;L^{2}(\mathbb{R}^n))$. Thus, we can extract a subsequence, which we still define as $\{u_{\varepsilon}\}$,  such that
\begin{equation}\label{strongly xi}
	u_{\varepsilon}^{\frac{m+1}{2}}\to \xi,\quad \mathrm{strongly\; in\; }C(0,T;L_{loc}^{2}(\mathbb{R}^n)),
\end{equation}
which implies
\begin{equation*}
	u_{\varepsilon}^{\frac{m+1}{2}}(x,t)\to \xi(x,t),\quad \mathrm{a.a.\;}x\in\mathbb{R}^n, t\in(0,T),
\end{equation*}
and then
\begin{equation}\label{aa xi 2/(m+1)}
	u_{\varepsilon}(x,t)\to \xi^{\frac{2}{m+1}}(x,t),\quad \mathrm{a.a.\;}x\in\mathbb{R}^n, t\in(0,T).
\end{equation}
Since it holds by Lebesgue dominated convergence theorem, \eqref{positivity and boundedness of u var}, \eqref{continuous u}, and \eqref{aa xi 2/(m+1)} that
\begin{equation}\label{strongly u L}
	u_{\varepsilon}\to \xi^{\frac{2}{m+1}}=u,\quad \mathrm{strongly\; in\; }L^\infty(0,T;L_{loc}^{1}(\mathbb{R}^n)),
\end{equation}
we have
\begin{equation}\label{aa u}
	u_{\varepsilon}(x,t)\to \xi^{\frac{2}{m+1}}(x,t)=u(x,t),\quad \mathrm{a.a.\;}x\in\mathbb{R}^n, t\in(0,T).
\end{equation}
Combining \eqref{strongly xi} with \eqref{aa u}, we obtain
\begin{equation}\label{strongly m+0.5}
	u_{\varepsilon}^{\frac{m+1}{2}}\to u^{\frac{m+1}{2}},\quad \mathrm{strongly\; in\; }C(0,T;L_{loc}^{2}(\mathbb{R}^n)).
\end{equation}
By virtue of Lemma \ref{Contraction of the difference of exponential function}, for $0\leq b\leq a$, $r\geq1$ and $0<m<1$,  we see
\begin{equation*}
	\begin{aligned}
		\vert a-b\vert^r\leq& \vert a^r-b^r\vert\\
		=&\vert (a^{\frac{m+1}{2}})^{\frac{2r}{m+1}}-(b^{\frac{m+1}{2}})^{\frac{2r}{m+1}}\vert\\
		\leq&\left(2^{\frac{2r}{m+1}-1}\cdot\frac{2r}{m+1}\right)\left((a^{\frac{m+1}{2}})^{\frac{2r}{m+1}-1}+(b^{\frac{m+1}{2}})^{\frac{2r}{m+1}-1}\right)\vert a^{\frac{m+1}{2}}-b^{\frac{m+1}{2}}\vert,
	\end{aligned}
\end{equation*}
which combined H$\mathrm{\ddot{o}}$lder's inequality, \eqref{positivity and boundedness of u var} with  \eqref{strongly m+0.5} can yield
\begin{equation*}
	\begin{aligned}
		&\sup_{0<t<T}\int\limits_{K}\vert u_{\varepsilon}(x,t)-u(x,t)\vert^r\\
		&\leq C\sup_{0<t<T}\int\limits_{K}\vert u_{\varepsilon}^{\frac{m+1}{2}}(x,t)-u^{\frac{m+1}{2}}(x,t)\vert dx\\
		&\leq C\sup_{0<t<T}\Vert u_{\varepsilon}^{\frac{m+1}{2}}(t)-u^{\frac{m+1}{2}}(t)\Vert_{L^2(K)}\vert K \vert^{\frac{1}{2}}\to0,
	\end{aligned}
\end{equation*}
for all compact sets $K$ as $\varepsilon\to 0$, where $C=C(\Vert u_0\Vert_{\infty},p,m,r)$.  Then for $1\leq r<\infty$, we can see
\begin{equation}
	u_{\varepsilon}\to u,\quad  \mathrm{strongly\; in\; }C(0,T;L_{loc}^{r}(\mathbb{R}^n)).\label{strongly u}
\end{equation}
In addition, from \eqref{positivity and boundedness of u var}, \eqref{approximating estimate 1} and \eqref{approximating estimate 2}, we obtain
\begin{align}
	&u_{\varepsilon}\to u,\quad  &&\mathrm{weakly\;star\; in\; }L^\infty(0,T;L^{2}(\mathbb{R}^n)\cap L^{\infty}(\mathbb{R}^n),\label{weakly u}\\
	&\nabla u_{\varepsilon}^{m}\to \nabla u^{m},\quad &&\mathrm{weakly\;star\; in\; }L^\infty(0,T;L^{2}(\mathbb{R}^n)).\label{weakly nable u m}
\end{align}
For any $\phi\in C^{\infty}(Q_T)$ with $\mathrm{supp}\;\phi(\cdot,t)\subset\subset\mathbb{R}^n$ and $0<t<T$, we see that $u_{\varepsilon}$ satisfies
\begin{equation*}
	\begin{aligned}
		&\int u_{\varepsilon}(t)\phi(t)dx-\int u_{0\varepsilon} \phi(0)dx-\int\limits_{0}^{t}\int u_{\varepsilon}(\tau)\partial_{\tau}\phi(\tau)dxd\tau\\
		&=-\int\limits_{0}^{t}\int \left(\nabla u_{\varepsilon}(\tau)^m\cdot\nabla\phi(\tau)+\varepsilon\nabla u_{\varepsilon}(\tau)\cdot\nabla\phi(\tau)+u_{\varepsilon}(\tau)^{p}\phi(\tau)\right)dxd\tau,
	\end{aligned}
\end{equation*}	
which yields \eqref{local existence equation} by \eqref{strongly u}-\eqref{weakly nable u m}. $\hfill\qedsymbol$

Next, we prove the decay property of the weak solution to \eqref{fast diffusion equation with absorption}-\eqref{initial data}.

\textbf{Proof of Theorem \ref{decay estimate}.}
Let $\eta_{R}(x)$ be the cut-off function defined in Lemma \ref{cut-off function}. We first show that the weak solution $u$ on $[0,T)$ given by Theorem \ref{local existence} satisfies \eqref{r norm estimation of u}.

For $r=1$, we multiply \eqref{approximating equation u} by $\eta_{R}$ and integrate it over $\mathbb{R}^n$, which yields that
\begin{equation}\label{1 norm estimation of u var}
		\begin{aligned}
			\frac{d}{dt}\int  u_{\varepsilon}\eta_{R}dx=-\int\nabla u_{\varepsilon}^{m} \nabla\eta_{R}dx-\varepsilon\int\nabla u_{\varepsilon}\nabla\eta_{R}dx-\int u_{\varepsilon}^{p}\eta_{R}dx.
		\end{aligned}
\end{equation}
 We obtain from \eqref{strongly u}-\eqref{1 norm estimation of u var} as the limit $\varepsilon\to0$ that
\begin{equation}\label{1 norm estimation of u}
	\begin{aligned}
		\frac{d}{dt}\int  u\eta_{R}dx=-\int\nabla u^{m} \nabla\eta_{R}dx-\int u^{p}\eta_{R}dx.
	\end{aligned}
\end{equation}
Since it holds by Lemma \ref{cut-off function} that
\begin{equation*}
	\vert \nabla\eta_{R}(x) \vert\leq \frac{C}{R},
\end{equation*}
we obtain by letting  $R\to\infty$ in \eqref{1 norm estimation of u} that
\begin{equation*}
	\begin{aligned}
		\frac{d}{dt}\int udx=-\int u^{p}dx\leq0.
	\end{aligned}
\end{equation*}
Then we have \eqref{r norm estimation of u} for $r=1$ on $0\leq t<T$.

For $r>1$, multiplying \eqref{approximating equation u} by  $u_{\varepsilon}^{r-1}\eta_{R}$ and integrating it over $\mathbb{R}^n$, we have
\begin{eqnarray} \label{r norm estimation of u 1}
&&\frac{1}{r}\frac{d}{dt}\int  u_{\varepsilon}^{r}\eta_{R}dx \notag\\
&&=-\frac{4m(r-1)}{(r+m-1)^2}\int\left\vert\nabla u_{\varepsilon}^{\frac{r+m-1}{2}}\right\vert^2 \eta_{R}dx \notag \\
&& \ \ \ -\frac{m}{r+m-1}\int \nabla u_{\varepsilon}^{r+m-1}\nabla\eta_{R}dx \notag \\
&& \ \ \ -\frac{4\varepsilon(r-1)}{r^2}\int\left\vert\nabla u_{\varepsilon}^{\frac{r}{2}}\right\vert^2 \eta_{R}dx-\frac{\varepsilon}{r}\int \nabla u_{\varepsilon}^{r}\nabla\eta_{R}dx-\int u_{\varepsilon}^{r+p-1}\eta_{R}dx \notag \\
&&\leq -\frac{m}{r+m-1}\int \nabla u_{\varepsilon}^{r+m-1}\nabla\eta_{R}dx \notag \\
& &\ \ \ -\frac{\varepsilon}{r}\int \nabla u_{\varepsilon}^{r}\nabla\eta_{R}dx-\int u_{\varepsilon}^{r+p-1}\eta_{R}dx,
\end{eqnarray}
which implies by \eqref{strongly u}-\eqref{weakly nable u m} as $\varepsilon\to0$ that
\begin{equation}\label{r norm estimation of u eta}
		\frac{1}{r}\frac{d}{dt}\int  u^{r}\eta_{R}dx\leq-\frac{m}{r+m-1}\int \nabla u^{r+m-1}\nabla\eta_{R}dx-\int u^{r+p-1}\eta_{R}dx.
\end{equation}
Let $R\to\infty$ in \eqref{r norm estimation of u eta}, we have
\begin{equation*}
	\frac{1}{r}\frac{d}{dt}\int  u^{r}dx=-\int u^{r+p-1}dx\leq0.
\end{equation*}
Thus, $\Vert u \Vert_r\leq \Vert u_0 \Vert_r$, which is also true for $r=\infty$. Then we have \eqref{r norm estimation of u} for $r>1$ on $0\leq t<T$.

Since $T<\infty$ is arbitrary taken, then we obtain \eqref{r norm estimation of u} for $t\geq0$.

Moreover, for $r>1$, we can derive by \eqref{r norm estimation of u 1} and Lemma \ref{cut-off function} that
\begin{eqnarray} \label{r norm estimation of u 2}
&&\frac{1}{r}\frac{d}{dt}\int  u_{\varepsilon}^{r}\eta_{R}dx \notag \\
&&=-\frac{4m(r-1)}{(r+m-1)^2}\int\left\vert\nabla u_{\varepsilon}^{\frac{r+m-1}{2}}\right\vert^2 \eta_{R}dx \notag \\
&& \ \ \ +\frac{m}{r+m-1}\int  u_{\varepsilon}^{r+m-1}\Delta\eta_{R}dx \notag\\
&&\ \ \ -\frac{4\varepsilon(r-1)}{r^2}\int\left\vert\nabla u_{\varepsilon}^{\frac{r}{2}}\right\vert^2 \eta_{R}dx+\frac{\varepsilon}{r}\int  u_{\varepsilon}^{r}\Delta\eta_{R}dx-\int u_{\varepsilon}^{r+p-1}\eta_{R}dx \notag\\
&&\leq -\frac{4m(r-1)}{(r+m-1)^2}\int\left\vert\nabla u_{\varepsilon}^{\frac{r+m-1}{2}}\right\vert^2 \eta_{R}dx \notag \\
&& \ \ \ +\frac{Cm}{(r+m-1)R^2}\int\limits_{\mathrm{supp}\;\eta_R}  u_{\varepsilon}^{r+m-1}dx \notag\\
&& \ \ \ -\frac{4\varepsilon(r-1)}{r^2}\int\left\vert\nabla u_{\varepsilon}^{\frac{r}{2}}\right\vert^2 \eta_{R}dx \notag \\
&& \ \ \ +\frac{C\varepsilon}{rR^2}\int\limits_{\mathrm{supp}\;\eta_R}  u_{\varepsilon}^{r}dx-\int u_{\varepsilon}^{r+p-1}\eta_{R}dx.
\end{eqnarray}
On the one hand, we can get by \eqref{r norm estimation of u 2} and H$\mathrm{\ddot{o}}$lder's inequality that
\begin{eqnarray} \label{r norm estimation of u 3}
&&\frac{1}{r}\frac{d}{dt}\int  u_{\varepsilon}^{r}\eta_{R}dx \notag\\
&&\leq \frac{C}{R^2}\int\limits_{\mathrm{supp}\;\eta_R}  u_{\varepsilon}^{r+m-1}dx+\frac{C\varepsilon}{R^2}\int\limits_{\mathrm{supp}\;\eta_R}  u_{\varepsilon}^{r}dx-\int u_{\varepsilon}^{r+p-1}\eta_{R}dx \notag \\
&&\leq \frac{C}{R^2}\int\limits_{\mathrm{supp}\;\eta_R}  u_{\varepsilon}^{r+m-1}dx+\frac{C\varepsilon}{R^2}\int\limits_{\mathrm{supp}\;\eta_R}  u_{\varepsilon}^{r}dx \notag \\
&&\ \ \ -\frac{\Vert u_{\varepsilon}\eta_R^{\frac{1}{r+p-1}} \Vert_{r}^{\frac{r(r+p-2)}{r-1}}}{\Vert u_{\varepsilon}\eta_R^{\frac{1}{r+p-1}} \Vert_{1}^{\frac{p-1}{r-1}}}.
\end{eqnarray}
Let $\varepsilon\to0$ and $R\to\infty$ in \eqref{r norm estimation of u 3}, we have
\begin{equation*}
	\frac{1}{r}\frac{d}{dt}\Vert u(t) \Vert_{r}^r\leq-\frac{\Vert u \Vert_{r}^{\frac{r(r+p-2)}{r-1}}}{\Vert u \Vert_{1}^{\frac{p-1}{r-1}}}\leq-\frac{\Vert u \Vert_{r}^{\frac{r(r+p-2)}{r-1}}}{\Vert u_{0} \Vert_{1}^{\frac{p-1}{r-1}}},
\end{equation*}	
which implies \eqref{t dacay estimation of u 1}. It is clear that \eqref{t dacay estimation of u 1} still hold for $r=1$.

Furthermore, for $r\geq3-m$, it follows by H$\mathrm{\ddot{o}}$lder's inequality, Nash inequality \eqref{Nash inequality} and Lemma \ref{cut-off function}  that
\begin{eqnarray} \label{r norm estimation of u 4}
&&\left\Vert  u_{\varepsilon}\eta_{R}^{\frac{1}{r+m-1}} \right\Vert_{r}^{\frac{r(n+2)}{n}} \notag \\
&&\leq \left\Vert  u_{\varepsilon}\eta_{R}^{\frac{1}{r+m-1}} \right\Vert_{\infty}^{\frac{(1-m)(n+2)}{n}}\left\Vert  u_{\varepsilon}\eta_{R}^{\frac{1}{r+m-1}} \right\Vert_{r+m-1}^{\frac{(r+m-1)(n+2)}{n}} \notag \\
&&\leq C\left\Vert  u_{\varepsilon}\eta_{R}^{\frac{1}{r+m-1}} \right\Vert_{\infty}^{\frac{(1-m)(n+2)}{n}}\left\Vert  u_{\varepsilon}\eta_{R}^{\frac{1}{r+m-1}} \right\Vert_{\frac{r+m-1}{2}}^{\frac{2(r+m-1)}{n}}\left\Vert \nabla \left(u_{\varepsilon}^{\frac{r+m-1}{2}}\eta_{R}^{\frac{1}{2}}\right)\right \Vert_{2}^{2} \notag \\
&&\leq C\left\Vert  u_{\varepsilon}\eta_{R}^{\frac{1}{r+m-1}} \right\Vert_{\infty}^{\frac{(1-m)(n+2)}{n}}\left\Vert  u_{\varepsilon}\eta_{R}^{\frac{1}{r+m-1}} \right\Vert_{1}^{\frac{2(r-m+1)}{n(r-1)}}\left\Vert  u_{\varepsilon}\eta_{R}^{\frac{1}{r+m-1}} \right\Vert_{r}^{\frac{2r(r+m-3)}{n(r-1)}} \notag \\
&& \ \ \ \times\left(\int\left\vert\nabla u_{\varepsilon}^{\frac{r+m-1}{2}}\right\vert^2 \eta_{R}dx+\frac{1}{4}\int u_{\varepsilon}^{r+m-1}\frac{\vert\nabla \eta_{R}\vert^2}{\eta_{R}} dx\right)  \notag \\
&&\leq C\left\Vert  u_{\varepsilon}\eta_{R}^{\frac{1}{r+m-1}} \right\Vert_{\infty}^{\frac{(1-m)(n+2)}{n}}\left\Vert  u_{\varepsilon}\eta_{R}^{\frac{1}{r+m-1}} \right\Vert_{1}^{\frac{2(r-m+1)}{n(r-1)}}\left\Vert  u_{\varepsilon}\eta_{R}^{\frac{1}{r+m-1}} \right\Vert_{r}^{\frac{2r(r+m-3)}{n(r-1)}} \notag\\
&& \ \ \ \times\left(\int\left\vert\nabla u_{\varepsilon}^{\frac{r+m-1}{2}}\right\vert^2 \eta_{R}dx+\frac{C}{4R^2}\int u_{\varepsilon}^{r+m-1}\eta_{R}^{1-2a} dx\right),
\end{eqnarray}
where $0<a<\frac{1}{2}$. Then we find by \eqref{r norm estimation of u 1} and \eqref{r norm estimation of u 4}  that

\begin{eqnarray} \label{r norm estimation of u 5}
&&\frac{1}{r}\frac{d}{dt}\int  u_{\varepsilon}^{r}\eta_{R}dx \notag \\
&&\leq \frac{C}{R^2}\int\limits_{\mathrm{supp}\;\eta_R}  u_{\varepsilon}^{r+m-1}dx+\frac{C\varepsilon}{R^2}\int\limits_{\mathrm{supp}\;\eta_R}  u_{\varepsilon}^{r}dx-\frac{4m(r-1)}{(r+m-1)^2}\int\left\vert\nabla u_{\varepsilon}^{\frac{r+m-1}{2}}\right\vert^2 \eta_{R}dx \notag \\
&&\leq \frac{C}{R^2}\int\limits_{\mathrm{supp}\;\eta_R}  u_{\varepsilon}^{r+m-1}dx+\frac{C\varepsilon}{R^2}\int\limits_{\mathrm{supp}\;\eta_R}  u_{\varepsilon}^{r}dx \notag \\
&& \ \ \ -\frac{4m(r-1)}{(r+m-1)^2}\Big(C\left\Vert  u_{\varepsilon}\eta_{R}^{\frac{1}{r+m-1}} \right\Vert_{r}^{\frac{r(nr-n-2m+4)}{n(r-1)}} \times \notag \\
&& \ \ \ \times \left\Vert  u_{\varepsilon}\eta_{R}^{\frac{1}{r+m-1}} \right\Vert_{\infty}^{-\frac{(1-m)(n+2)}{n}}\left\Vert  u_{\varepsilon}\eta_{R}^{\frac{1}{r+m-1}} \right\Vert_{1}^{-\frac{2(r-m+1)}{n(r-1)}}\Big),
\end{eqnarray}
which yields by letting $\varepsilon\to0$ and $R\to\infty$ in \eqref{r norm estimation of u 5}, we have
\begin{align*}
	&\frac{1}{r}\frac{d}{dt}\Vert u(t) \Vert_{r}^r\\
	&\leq-\frac{4Cm(r-1)}{(r+m-1)^2}\Vert u \Vert_{r}^{\frac{r(nr-n-2m+4)}{n(r-1)}}\Vert u(t) \Vert_{\infty}^{-\frac{(1-m)(n+2)}{n}}\Vert u(t) \Vert_{1}^{-\frac{2(r-m+1)}{n(r-1)}}\\
	&\leq-\frac{4Cm(r-1)}{(r+m-1)^2}\Vert u \Vert_{r}^{\frac{r(nr-n-2m+4)}{n(r-1)}}\Vert u_{0} \Vert_{\infty}^{-\frac{(1-m)(n+2)}{n}}\Vert u_{0} \Vert_{1}^{-\frac{2(r-m+1)}{n(r-1)}}.
\end{align*}	
Then we have
\begin{equation}\label{t norm estimation of u 6}
	\Vert u(t) \Vert_{r}\leq C\left(1+t\right)^{-\frac{n(r-1)}{2r(2-m)}}, \quad t\geq0,
\end{equation}
where $C=C(\Vert u_0 \Vert_{1},\Vert u_0 \Vert_{\infty},n,r,m).$

Combining \eqref{t dacay estimation of u 1} with \eqref{t norm estimation of u 6}, we can have \eqref{t dacay estimation of u 2}. $\hfill\qedsymbol$

\begin{remark} For $r>3-m$, there exists a $n_0>\frac{2(2-m)}{p-1}$ such that
for $n>n_0$, we can see that \eqref{t dacay estimation of u 2} degenerates into \eqref{t norm estimation of u 6}.
\end{remark}

\section{Asymptotic convergence}
In order to study the asymptotic convergence of the solution to \eqref{fast diffusion equation with absorption}-\eqref{initial data}, we adopt the following change of similarity variables:
\begin{equation}\label{rescaled variables}
	x=yR(t),\;R(t)=(1+\lambda t)^{1/\lambda},\;s=\frac{1}{\lambda}\ln(1+\lambda t),
\end{equation}
where $\lambda=2-n(1-m)$. And the new unknown function $\rho=\rho(y,s)$ is given by
\begin{equation}\label{rho variables}
	\rho(y,s)=R(t)^{n}u(x,t),
\end{equation}
which leads \eqref{fast diffusion equation with absorption}-\eqref{initial data} to
\begin{equation}\label{rho equation}
	\rho_s=\text{div}_{y}(\rho y+\nabla_{y}\rho^{m})-e^{-\delta s}\rho^p,\quad y\in \mathbb{R}^n,~s>0,
\end{equation}
with initial data
\begin{equation}\label{rho initial data}
	\rho(y,0)=\rho_{0}(y)=u_0(x),\quad y\in \mathbb{R}^n,
\end{equation}
where $\delta=np-nm-2$. In this section, we consider the diffusion-dominanted regime
\begin{equation*}\label{diffusion dominance}
	\delta=np-nm-2>0,\;\text{i.e.}\;p>m+\frac{2}{n},
\end{equation*}
where the absorption becomes small perturbation in determining the long-time behavior.
 In addition, we can directly derive the following corollary of the Theorem \ref{decay estimate}, which we omit the proof here.

\begin{corollary}\label{decay estimate of rho} Let the Assumption hold. Suppose $0<m<1$ and $p>1$. Then for $1\leq r \leq \infty$ and $s_0>0$, the global weak solution $\rho(y,s)$ to \eqref{rho equation}-\eqref{rho initial data} satisfying
\begin{equation*}\label{r norm estimation of rho}
	\Vert \rho(s) \Vert_{r}\leq e^{\frac{n(r-1)}{r}s}\Vert u_0 \Vert_{r}, \quad s\geq0,
\end{equation*}	
\begin{equation*}\label{dacay estimation of rho 1}
	\Vert \rho(s) \Vert_{r}\leq C(\Vert u_0 \Vert_{1},\Vert u_0 \Vert_{\infty},r,p,m), \quad s>s_0,
\end{equation*}
and
\begin{equation*}\label{dacay estimation of rho 2}
	\Vert \rho(s) \Vert_{r}\leq C(\Vert u_0 \Vert_{1},\Vert u_0 \Vert_{\infty},r,p,m)s^{-\frac{r-1}{r(p-1)}}, \quad 0<s<s_0.
\end{equation*}
Moreover, for the case of $r\geq3-m$, we have
\begin{equation*}\label{dacay estimation of rho 3}
	\Vert \rho(s) \Vert_{r}\leq C(\Vert u_0 \Vert_{1},\Vert u_0 \Vert_{\infty},n,r,p,m), \quad s>s_0,
\end{equation*}
and
\begin{equation*}\label{dacay estimation of rho 4}
	\Vert \rho(s) \Vert_{r}\leq C(\Vert u_0 \Vert_{1},\Vert u_0 \Vert_{\infty},n,r,p,m)s^{-\max{\{\frac{r-1}{r(p-1)},\frac{n(r-1)}{2r(2-m)}}\}}, \quad 0<s<s_0.
\end{equation*}
\end{corollary}
%\begin{proof} It is directly to get \eqref{r norm estimation of rho} by \eqref{r norm estimation of u} and \eqref{rho variables}. For $s>0$, by virtue of \eqref{t dacay estimation of u 1} and \eqref{rho variables}, we have
%\begin{equation*}
%	\begin{aligned}
%		\Vert \rho(s) \Vert_{r}=&e^{\frac{n(r-1)}{r}s}\Vert u \Vert_{r}\\
%		\leq&C(\Vert u_0 \Vert_{1},\Vert u_0 \Vert_{\infty},r,p)\left(\frac{e^{\lambda s}-1}{\lambda}\right)^{-\frac{r-1}{r(p-1)}}\left(e^{-\lambda s}\right)^{-\frac{r-1}{r(p-1)}}\\
%		=&C(\Vert u_0 \Vert_{1},\Vert u_0 \Vert_{\infty},r,p)\left(\frac{1-e^{-\lambda s}}{\lambda}\right)^{-\frac{r-1}{r(p-1)}}.
%	\end{aligned}
%\end{equation*}	
%If $s>s_0$, then
%\begin{equation*}
%	\frac{1-e^{-\lambda s}}{\lambda}>\frac{1-e^{-\lambda s_0}}{\lambda}
%\end{equation*}	
%implies	
%\begin{equation*}
%		\Vert \rho(s) \Vert_{r}\leq C(\Vert u_0 \Vert_{1},\Vert u_0 \Vert_{\infty},r,p)\left(\frac{1-e^{-\lambda s_0}}{\lambda}\right)^{-\frac{r-1}{r(p-1)}}\leq C(\Vert u_0 \Vert_{1},\Vert u_0 \Vert_{\infty},r,p,m).
%\end{equation*}	
%If $0<s<s_0$, then
%\begin{equation*}
%	\frac{1-e^{-\lambda s}}{\lambda}>\frac{s(1-e^{-\lambda_0 s})}{\lambda s_0}
%\end{equation*}	
%implies	
%\begin{equation*}
%	\Vert \rho(s) \Vert_{r}\leq C(\Vert u_0 \Vert_{1},\Vert u_0 \Vert_{\infty},r,p)\left(\frac{s(1-e^{-\lambda_0 s})}{\lambda s_0}\right)^{-\frac{r-1}{r(p-1)}}\leq C(\Vert u_0 \Vert_{1},\Vert u_0 \Vert_{\infty},r,p,m)s^{-\frac{r-1}{r(p-1)}},
%\end{equation*}	
%where $0<\lambda_0<\lambda$.
%Similarly, we can have \eqref{dacay estimation of rho 3} and \eqref{dacay estimation of rho 4}.		
%\end{proof}

Let us define the total mass of $\rho(x,s)$ at time $s$ by
\[
M(s):=\int\rho(y,s)dy.
\]
Then we have the following lemma.

\begin{lemma}\label{mass} Let the Assumption hold. Suppose $\frac{n-2}{n}<m<1$.
	Then the mass $M(s)$ is decreasing and the limiting mass $M_{\infty}=M(s=\infty)$ is strictly positive.
\end{lemma}

\begin{proof} Let $\eta_{R}(y)$ be the cut-off function defined by Lemma \ref{cut-off function}. We consider the approximated problem of \eqref{rho equation}-\eqref{rho initial data} as follows
	\begin{equation}\label{rho var equation}
		\partial_s\rho_{\varepsilon}=\text{div}_{y}(\rho_{\varepsilon} y+\nabla_{y}\rho_{\varepsilon}^{m}+\varepsilon e^{(\lambda-2)s}\nabla_{y}\rho_{\varepsilon})-e^{-\delta s}\rho_{\varepsilon}^p,\quad y\in \mathbb{R}^n,~s>0,
	\end{equation}
	\begin{equation}\label{rho var initial data}
		\rho_{\varepsilon}(y,0)=(\rho_{0}*\widetilde{\eta}_{\varepsilon})(y)=\rho_{0\varepsilon},\quad y\in \mathbb{R}^n,
	\end{equation}	
	for $\varepsilon>0$.
	We multiply \eqref{rho var equation} by $\eta_{R}$ and integrate it over $\mathbb{R}^n$, which yields that
	\begin{eqnarray}\label{mass 1}
			&&\frac{d}{ds}\int\rho_{\varepsilon} \eta_{R}dy \notag \\
			&&= -\int(\rho_{\varepsilon} y+\nabla_{y}\rho_{\varepsilon}^{m}+\varepsilon e^{(\lambda-2)s}\nabla_{y}\rho_{\varepsilon})\nabla\eta_{R}dy-e^{-\delta s}\int\rho_{\varepsilon}^p \eta_{R}dy.	
	\end{eqnarray}
	By \eqref{strongly u} and \eqref{weakly nable u m}, we have
	\begin{align}
		&\rho_{\varepsilon}\to \rho,\quad && \mathrm{strongly\; in\; }C(0,T;L_{loc}^{r}(\mathbb{R}^n)),&\label{stongly rho r norm}\\
		&\nabla \rho_{\varepsilon}^{m}\to \nabla \rho^{m},\quad && \mathrm{weakly\;star\; in\; }L^\infty(0,T;L^{2}(\mathbb{R}^n)),&\label{weakly nable rho m}
	\end{align}
for $r\geq1$. Thus, let $\varepsilon\to0$ and $R\to\infty$ in \eqref{mass 1}, we have
\begin{align*}
	\frac{d}{ds}M(s)=-e^{-\delta s}\int\rho^{p}dy\geq 0,
\end{align*}
which yields that the mass $M(s)$ is decreasing. Moreover, combined Corollary \ref{decay estimate of rho} with $p>m+\frac{2}{n}>1$, we have
	\begin{align*}
		\frac{d}{ds}M(s)=-e^{-\delta s}\int\rho^{p}dy\geq -C(\Vert u_0 \Vert_{1},\Vert u_0 \Vert_{\infty},n,p,m)e^{-\delta s}M(s),
	\end{align*}	
	which yields
	\begin{equation}\label{M > 0}
		M(s)\geq M_{\infty}\geq M(0)e^{-C(\Vert u_0 \Vert_{1},\Vert u_0 \Vert_{\infty},n,p,m)\int_{0}^{\infty}e^{-\delta s}ds}>0.
	\end{equation}
This completes the proof.
\end{proof}

Next, we will begin to prove Theorem \ref{asymptotic behavior}, which is divided to some steps.

\subsection{Uniform bound of the second moment}
First, we prove the uniformly bound of the second moment of the weak solution $\rho(y,s)$ to \eqref{rho equation}-\eqref{rho initial data} by the generalized Shanon's inequality \eqref{Generalized Shannon's inequality}.
\begin{proposition}\label{bound of the second moment} Let the Assumption hold. Suppose $u_{0}\in L_{2}^1(\mathbb{R}^n)$, $\frac{n}{n+2}<m<1$ and $p>m+\frac{2}{n}$,
	For the weak solution $\rho(y,s)$ to \eqref{rho equation}-\eqref{rho initial data},  there exists $C>0$ such that
\begin{equation}\label{second moment}
	\int\rho \vert y \vert^2dy\leq C,
\end{equation}
where $C=C\left(\int u_0 \vert x \vert^2dx,\Vert u_0 \Vert_{1},n,m,p\right)$.	
\end{proposition}
\begin{proof} Let $\eta_{R}(y)$ be the cut-off function defined by Lemma \ref{cut-off function}.
We multiply \eqref{rho var equation} by $\vert y \vert^2\eta_{R}$ and integrate it over $\mathbb{R}^n$, which yields that
\begin{eqnarray}\label{second moment 1}
	&&\frac{d}{ds}\int\rho_{\varepsilon} \vert y \vert^2\eta_{R}dy \notag \\
	&&=-\int(\rho_{\varepsilon} y+\nabla_{y}\rho_{\varepsilon}^{m}+\varepsilon e^{(\lambda-2)s}\nabla_{y}\rho_{\varepsilon})(2y\eta_{R}+\vert y \vert^2\nabla\eta_{R})dy \notag \\
&& \ \ \ -e^{-\delta s}\int\rho_{\varepsilon}^p \vert y \vert^2\eta_{R}dy.	
\end{eqnarray}
By \eqref{stongly rho r norm}-\eqref{weakly nable rho m}, letting $\varepsilon\to0$ and $R\to\infty$ in \eqref{second moment 1}, we have
\begin{equation}\label{second moment 2}
	\frac{d}{ds}\int\rho \vert y \vert^2dy=2n\int\rho^mdy-2\int\rho \vert y \vert^2dy-e^{-\delta s}\int\rho^p \vert y \vert^2dy.
\end{equation}
Since it holds by the generalized Shanon's inequality \eqref{Shannon's inequality} and Corollary \ref{decay estimate of rho} that
\begin{equation}\label{second moment 3}
	\int\rho^mdy\leq C\left(\int\rho \vert y \vert^2dy\right)^{m\sigma},
\end{equation}
substituting \eqref{second moment 3} into \eqref{second moment 2}, we have
\begin{equation}\label{second moment 4}
	\frac{d}{ds}\int\rho \vert y \vert^2dy\leq C\left(\int\rho \vert y \vert^2dy\right)^{m\sigma}-2\int\rho \vert y \vert^2dy,
\end{equation}
where $\sigma=\frac{n(1-m)}{2m}$ and $C=C(\Vert u_0 \Vert_{1},n,m,p)$.
Thus, we obtain
\begin{align*}
	\int\rho \vert y \vert^2dy\leq& e^{-2(1-m\sigma)s}\left(\left(\int \rho_0 \vert y \vert^2dy\right)^{1-m\sigma}+C(1-m\sigma)s\right)^{\frac{1}{1-m\sigma}}\\
	\leq&C\left(\int u_0 \vert x \vert^2dx,\Vert u_0 \Vert_{1},n,m,p\right),
\end{align*}
since $1-m\sigma>0$ for $\frac{n}{n+2}<m<1$.
\end{proof}

\subsection{Entropy production}
Next, we derive the entropy production of the \eqref{rho equation}-\eqref{rho initial data}.

\begin{proposition} Let the Assumption hold. For the weak solution $\rho(y,s)$ to \eqref{rho equation}-\eqref{rho initial data}, define the entropy energy
\begin{equation}\label{entropy rho}
		E(\rho)(s)=\frac{1}{m-1}\int\rho^{m}dy+\frac{1}{2}\int\vert y\vert^{2}\rho dy.
\end{equation}
Then, for $u_{0}\in L_{2}^1(\mathbb{R}^n)$, $\frac{n}{n+2}<m<1$ and $p>m+\frac{2}{n}$, it holds that
\begin{eqnarray}\label{entropy estimate rho }
	&&E(\rho)(s)-E(\rho_{0})\notag \\
&&\leq -\int\limits_{0}^{s}\int\rho(\tau)\left\vert\nabla\left(\frac{m}{m-1}\rho(\tau)^{m-1}+\frac{1}{2}\vert y\vert^{2}\right)\right\vert^{2}dyd\tau \notag \\
&& \ \ \ -\int\limits_{0}^{s}e^{-\delta s}\int\rho(\tau)^{p}\left(\frac{m}{m-1}\rho(\tau)^{m-1}+\frac{1}{2}\vert y\vert^{2}\right)dyd\tau.
\end{eqnarray}
\end{proposition}

\begin{proof}  For $R_1<R_2$, let $\eta_{R_1}(y)$ and $\eta_{R_2}(y)$ be the smooth cut-off function defined by Lemma \ref{cut-off function} with $R=R_1$ and $R=R_2$, respectively. We set
\begin{align*}
	A_{\varepsilon}\equiv\frac{m}{m-1}\rho_{\varepsilon}^{m-1}+\frac{1}{2}\vert y\vert^{2}+\varepsilon e^{(\lambda-2)s}(\log\rho_{\varepsilon}+1),\\
	\tilde{A}_{\varepsilon}\equiv\frac{m}{m-1}\rho_{\varepsilon}^{m-1}\eta_{R_1}+\frac{1}{2}\vert y\vert^{2}\eta_{R_1}+\varepsilon e^{(\lambda-2)s}(\log\rho_{\varepsilon}+1)\eta_{R_1}.
\end{align*}	
Then \eqref{rho var equation} can be converted to
\begin{equation*}
	\partial_s\rho_{\varepsilon}=\text{div}_{y}(\rho_{\varepsilon}\nabla_{y}A_{\varepsilon})-e^{-\delta s}\rho_{\varepsilon}^p,
\end{equation*}
which, by multiplying by $\tilde{A}_{\varepsilon}\eta_{R_2}$ and integrating with respect to $y$ over $\mathbb{R}^n$,  can yield
\begin{equation}\label{rho var equation 2}
	\int\partial_s\rho_{\varepsilon }\tilde{A}_{\varepsilon}\eta_{R_2}dy=-\int\rho_{\varepsilon}\nabla_{y}A_{\varepsilon}\nabla_{y}(\tilde{A}_{\varepsilon}\eta_{R_2})dy-e^{-\delta s}\int\rho_{\varepsilon}^p\tilde{A}_{\varepsilon}\eta_{R_2}dy.
\end{equation}
Since
\begin{equation*}
	A_{\varepsilon}=A_{\varepsilon}(1-\eta_{R_1})+\tilde{A}_{\varepsilon},
\end{equation*}
the first term of the right hand side in \eqref{rho var equation 2} is
\begin{eqnarray} \label{A var 1}
&&-\int\rho_{\varepsilon}\nabla_{y}A_{\varepsilon}\nabla_{y}(\tilde{A}_{\varepsilon}\eta_{R_2})dy  \notag \\	 &&=-\int\rho_{\varepsilon}\nabla_{y}(A_{\varepsilon}(1-\eta_{R_1}))\nabla_{y}\tilde{A}_{\varepsilon}\eta_{R_2}dy-\int\rho_{\varepsilon}\vert\nabla_{y}\tilde{A}_{\varepsilon}\vert^2\eta_{R_2} dy  \notag \\
&&\ \ \ -\int\rho_{\varepsilon}\nabla_{y}A_{\varepsilon}\tilde{A}_{\varepsilon}\nabla_{y}\eta_{R_2}dy.
\end{eqnarray}
Moreover, the  first term of the right hand side in \eqref{A var 1} is
\begin{eqnarray}\label{A var 2}
&&-\int\rho_{\varepsilon}\nabla_{y}(A_{\varepsilon}(1-\eta_{R_1}))\nabla_{y}\tilde{A}_{\varepsilon}\eta_{R_2}dy \notag \\
&&=-\int\rho_{\varepsilon}\vert\nabla_{y}A_{\varepsilon}\vert^2\eta_{R_1}(1-\eta_{R_1})\eta_{R_2}dy \notag \\
&& \ \ \ -\frac{1}{2}\int\rho_{\varepsilon}\nabla_{y}\vert A_{\varepsilon}\vert^2\nabla_{y}\eta_{R_1}(1-\eta_{R_1})\eta_{R_2}dy \notag \\
&& \ \ \ +\int\rho_{\varepsilon}A_{\varepsilon}\nabla_{y}\tilde{A}_{\varepsilon}\nabla_{y}\eta_{R_1}\eta_{R_2}dy.
\end{eqnarray}
On the other hand, the left hand side in \eqref{rho var equation 2} is
\begin{eqnarray}\label{s derivative of E var}
&&\int\partial_s\rho_{\varepsilon}\tilde{A}_{\varepsilon}\eta_{R_2}dy \notag \\
&&=\frac{d}{ds}\int\left(\frac{1}{m-1}\rho_{\varepsilon}^{m}+\frac{1}{2}\vert y\vert^{2}\rho_{\varepsilon}+\varepsilon e^{(\lambda-2)s}\rho_{\varepsilon}\log\rho_{\varepsilon}\right)\eta_{R_1}\eta_{R_2} dy.
\end{eqnarray}
Substituting \eqref{A var 1}, \eqref{A var 2} and \eqref{s derivative of E var} into \eqref{rho var equation 2}, and integrating \eqref{rho var equation 2} in $s$, we obtain
\begin{eqnarray}\label{E estimate 1}
&&\int\left(\frac{1}{m-1}\rho_{\varepsilon}^{m}+\frac{1}{2}\vert y\vert^{2}\rho_{\varepsilon}+\varepsilon e^{(\lambda-2)s}\rho_{\varepsilon}\log\rho_{\varepsilon}\right)\eta_{R_1}\eta_{R_2} dy \notag \\
&&=\int\left(\frac{1}{m-1}\rho_{0\varepsilon}^{m}+\frac{1}{2}\vert y\vert^{2}\rho_{0\varepsilon}+\varepsilon e^{(\lambda-2)s}\rho_{0\varepsilon}\log\rho_{0\varepsilon}\right)\eta_{R_1}\eta_{R_2} dy \notag \\
&& \ \ \ -\int\limits_{0}^{s}e^{-\delta s}\int\rho_{\varepsilon}^p\tilde{A}_{\varepsilon}\eta_{R_2}dyd\tau
-\int\limits_{0}^{s}\int\rho_{\varepsilon}\vert\nabla_{y}A_{\varepsilon}\vert^2\eta_{R_1}(1-\eta_{R_1})\eta_{R_2}dy\tau \notag \\
&& \ \ \ -\frac{1}{2}\int\limits_{0}^{s}\int\rho_{\varepsilon}\nabla_{y}\vert A_{\varepsilon}\vert^2\nabla_{y}\eta_{R_1}(1-\eta_{R_1})\eta_{R_2}dyd\tau \notag \\
&&\ \ \ +\int\limits_{0}^{s}\int\rho_{\varepsilon}A_{\varepsilon}\nabla_{y}\tilde{A}_{\varepsilon}\nabla_{y}\eta_{R_1}\eta_{R_2}dyd\tau  \notag\\
&& \ \ \ -\int\limits_{0}^{s}\int\rho_{\varepsilon}\vert\nabla_{y}\tilde{A}_{\varepsilon}\vert^2\eta_{R_2} dyd\tau-\int\limits_{0}^{s}\int\rho_{\varepsilon}\nabla_{y}A_{\varepsilon}\tilde{A}_{\varepsilon}\nabla_{y}\eta_{R_2}dyd\tau,
\end{eqnarray}
which yields by letting $R_2\to\infty$ that
\begin{eqnarray}\label{E estimate 2}
&&\int\left(\frac{1}{m-1}\rho_{\varepsilon}^{m}+\frac{1}{2}\vert y\vert^{2}\rho_{\varepsilon}+\varepsilon e^{(\lambda-2)s}\rho_{\varepsilon}\log\rho_{\varepsilon}\right)\eta_{R_1} dy \notag\\
&&\leq \int\left(\frac{1}{m-1}\rho_{0\varepsilon}^{m}+\frac{1}{2}\vert y\vert^{2}\rho_{0\varepsilon}+\varepsilon e^{(\lambda-2)s}\rho_{0\varepsilon}\log\rho_{0\varepsilon}\right)\eta_{R_1} dy \notag \\
&& \ \ \ -\int\limits_{0}^{s}e^{-\delta s}\int\rho_{\varepsilon}^p\tilde{A}_{\varepsilon}dyd\tau-\int\limits_{0}^{s}\int\rho_{\varepsilon}\vert\nabla_{y}A_{\varepsilon}\vert^2\eta_{R_1}(1-\eta_{R_1})dyd\tau \notag \\
&& \ \ \ -\frac{1}{2}\int\limits_{0}^{s}\int\rho_{\varepsilon}\nabla_{y}\vert A_{\varepsilon}\vert^2\nabla_{y}\eta_{R_1}(1-\eta_{R_1})dyd\tau+\int\limits_{0}^{s}\int\rho_{\varepsilon}A_{\varepsilon}\nabla_{y}\tilde{A}_{\varepsilon}\nabla_{y}\eta_{R_1}dyd\tau \notag \\
&& \ \ \ -\int\limits_{0}^{s}\int\rho_{\varepsilon}\vert\nabla_{y}\tilde{A}_{\varepsilon}\vert^2 dyd\tau.
\end{eqnarray}
By Theorem \ref{local existence}, Lemma \ref{Generalized Shannon's inequality} and Proposition \ref{bound of the second moment}, we have
\begin{align*}
	&\rho_{\varepsilon}\to \rho,\quad && \mathrm{weakly\;star\; in\; }L^\infty(0,T;L^{m}(\mathbb{R}^n),&\\
	&\nabla \rho_{\varepsilon}^{m-\frac{1}{2}}\to \nabla \rho^{m-\frac{1}{2}},\quad &&\mathrm{weakly\; in\; }L^2(0,T;L^{2}(\mathbb{R}^n)),&
\end{align*}
which, combined with \eqref{stongly rho r norm} and \eqref{weakly nable rho m},  can yield by passing $\varepsilon\to0$ that
\begin{equation*}
	\begin{aligned}
		&\int\left(\frac{1}{m-1}\rho^{m}+\frac{1}{2}\vert y\vert^{2}\rho\right)\eta_{R_1} dy-\int\left(\frac{1}{m-1}\rho_{0}^{m}+\frac{1}{2}\vert y\vert^{2}\rho_{0}\right)\eta_{R_1} dy\\
		&\leq-\int\limits_{0}^{s}\int\rho\left\vert\nabla_{y}\left(\frac{m}{m-1}\rho_{\varepsilon}^{m-1}+\frac{1}{2}\vert y\vert^{2}\right)\right\vert^2\eta(1-\eta_{R_1})dyd\tau\\
		& \ \ \ -\int\limits_{0}^{s}e^{-\delta \tau}\int\rho^p\left(\frac{m}{m-1}\rho_{\varepsilon}^{m-1}\eta_{R_1}+\frac{1}{2}\vert y\vert^{2}\eta_{R_1}\right)dyd\tau\\
		& \ \ \ -\frac{1}{2}\int\limits_{0}^{s}\int\rho\nabla_{y}\left\vert \frac{m}{m-1}\rho_{\varepsilon}^{m-1}+\frac{1}{2}\vert y\vert^{2}\right\vert^2\nabla_{y}\eta_{R_1}(1-\eta_{R_1})dyd\tau\\
		&\ \ \ +\int\limits_{0}^{s}\int\rho \left(\frac{m}{m-1}\rho_{\varepsilon}^{m-1}+\frac{1}{2}\vert y\vert^{2}\right)\nabla_{y}\left(\frac{m}{m-1}\rho_{\varepsilon}^{m-1}\eta_{R_1}+\frac{1}{2}\vert y\vert^{2}\eta_{R_1}\right)\nabla_{y}\eta_{R_1} dyd\tau\\
		&\ \ \ -\int\limits_{0}^{s}\int\rho\left\vert\nabla_{y}\left(\frac{m}{m-1}\rho_{\varepsilon}^{m-1}\eta_{R_1}+\frac{1}{2}\vert y\vert^{2}\eta_{R_1}\right)\right\vert^2 dyd\tau,
	\end{aligned}
\end{equation*}
which yields \eqref{entropy estimate rho } for $R_1\to\infty$.
\end{proof}

\subsection{Properties of the relative entropy}
Then we will show two properties of the relative entropy:
\begin{equation}\label{relative entropy}
	E(f|g)=E(f(y))-E(g(y)).
\end{equation}
\begin{proposition}[\cite{upper bound of the relative entropy}] \label{upper bound of the relative entropy 1}Let $m>\frac{n-1}{n}$ and $m\neq1$. Then, for any non-negative function $f\in L_{2}^1(\mathbb{R}^n)$ such that $\nabla f^{m-\frac{1}{2}}\in L^2(\mathbb{R}^n)$ and $\Vert f \Vert_1=M$, it holds that
	\begin{equation}\label{upper bound of the relative entropy 2}
		E(f|\rho_{M(s)})\leq\frac{1}{2}\int f(y)\left\vert\frac{m}{m-1}\nabla f(y)^{m-1}+y\right\vert^{2}dy.
	\end{equation}	
\end{proposition}
It is clear that Proposition \ref{upper bound of the relative entropy 1} implies the upper bound of the relative entropy. Moreover, we consider the lower bound of the relative entropy by Lemma \ref{CK inequality 1} with $f=\rho$ and $g=\rho_{M(s)}$. Since it holds by the decreasing of $M(s)$ and the increasing of $\theta_{M(s)}$ that
\begin{equation*}
	\Vert\rho_{M(s)}\Vert_{2-m}\leq \Vert\rho_{M(s)}\Vert_{\infty}^{1-m}\Vert\rho_{M(s)}\Vert_{1}\leq\frac{\gamma}{\theta_{M(0)}}M(0),
\end{equation*}
for $\frac{n-2}{n}<m<1$, we have $\rho_{M(s)}\in L^{2-m}(\mathbb{R}^n)$. In addition,
\begin{equation*}
	\frac{1}{2}\vert y\vert^{2}= \frac{m}{1-m}\left(\rho_{M(s)}^{m-1}-\frac{\theta_{M(s)}}{\gamma}\right).
\end{equation*}
Thus, we have
\begin{align*}
	D(\rho|\rho_{M(s)})&=\frac{1}{m-1}\int \rho^{m}dy-\frac{1}{m-1}\int \rho_{M(s)}^{m}dy \\
& \ \ \ -\frac{m}{m-1}\int \rho_{M(s)}^{m-1}(\rho-\rho_{M(s)})dy\\
	&=\frac{1}{m-1}\int \rho^{m}dy-\frac{1}{m-1}\int \rho_{M(s)}^{m}dy+\frac{1}{2}\int\vert y\vert^{2}(\rho-\rho_{M(s)})dy\\
	&=E(\rho|\rho_{M(s)}),
\end{align*}
which yields by Csisz$\mathrm{\acute{a}}$r-Kullback inequality \eqref{CK inequality 2} that
\begin{equation}\label{CK inequality 3}
	\Vert \rho-\rho_{M(s)} \Vert_{1}\leq C\sqrt{E(\rho|\rho_{M(s)})},
\end{equation}
where
\begin{align*}
	C&=M(s)^{\frac{m-2}{2}}\left(\frac{2}{m}\int \rho_{M(s)}^{2-m}dx\right)^{\frac{1}{2}}\\
	&\leq \left(M(0)e^{-C(\Vert u_0 \Vert_{1},\Vert u_0 \Vert_{\infty},n,p,m)\int_{0}^{\infty}e^{-\delta s}ds}\right)^{\frac{m-2}{2}}\left(\frac{2}{m}\int \rho_{M(s)}^{2-m}dx\right)^{\frac{1}{2}}\\
	&<\infty,
\end{align*}
since \eqref{M > 0} and  $\frac{n-2}{n}<\frac{n-1}{n}<m<1$.

\subsection{Asymptotic convergence}
Let us consider
\begin{eqnarray}\label{E rho M 1}
\frac{d}{ds}E(\rho_{M(s)})&=&\frac{d}{ds}\int \rho_{M(s)}\left(\frac{1}{m-1}\rho_{M(s)}^{m-1}+\frac{1}{2}\vert y\vert^{2}\right)dy \notag\\
		&=&\int\partial_{s}\rho_{M(s)}\left(\frac{m}{m-1}\rho_{M(s)}^{m-1}+\frac{1}{2}\vert y\vert^{2}\right)dy.
\end{eqnarray}
Note that
\begin{equation}\label{E rho M 2}
	\frac{m}{m-1}\rho_{M(s)}^{m-1}+\frac{1}{2}\vert y\vert^{2}=-\frac{1}{2}\theta_{M(s)},
\end{equation}
then we have, from \eqref{E rho M 1} and \eqref{E rho M 2}, that
\begin{eqnarray}\label{E rho M 3}
		\frac{d}{ds}E(\rho_{M(s)})&=&-\frac{1}{2}\theta_{M(s)}\int\partial_{s}\rho_{M(s)}dy \notag\\
		&=&-\frac{1}{2}\theta_{M(s)}\frac{d}{ds}\int\rho_{M(s)}dy \notag\\
		&=&-\frac{1}{2}\theta_{M(s)}\frac{d}{ds}M(s) \notag\\
		&=&\frac{1}{2}\theta_{M(s)}e^{-\delta s}\int\rho^p dy \notag\\
		&=&-e^{-\delta s}\int\rho^p \left(\frac{m}{m-1}\rho_{M(s)}^{m-1}+\frac{1}{2}\vert y\vert^{2}\right)dy.
\end{eqnarray}
On the other hand, by virtue of \eqref{entropy estimate rho }, we have
\begin{eqnarray}\label{E 1}
		\frac{d}{ds}E(\rho)&\leq &-\int\rho\left\vert\nabla\left(\frac{m}{m-1}\rho^{m-1}+\frac{1}{2}\vert y\vert^{2}\right)\right\vert^{2}dy \notag\\
		&& -e^{-\delta s}\int\rho^{p}\left(\frac{m}{m-1}\rho^{m-1}+\frac{1}{2}\vert y\vert^{2}\right)dy.
\end{eqnarray}

Combining \eqref{E rho M 3} and \eqref{E 1}, we obtain by Mean-Value Theorem that
\begin{eqnarray}\label{E E rho M 1}
&&\frac{d}{ds}E(\rho|\rho_{M(s)}) \notag\\
&&	\leq-\int\rho\left\vert\nabla\left(\frac{m}{m-1}\rho^{m-1}+\frac{1}{2}\vert y\vert^{2}\right)\right\vert^{2}dy \notag \\
&& \ \ \ -\frac{m}{m-1}e^{-\delta s}\int\rho^{p}\left(\rho^{m-1}-\rho_{M(s)}^{m-1}\right)dy  \notag\\
&&\leq -\int\rho\left\vert\nabla\left(\frac{m}{m-1}\rho^{m-1}+\frac{1}{2}\vert y\vert^{2}\right)\right\vert^{2}dy \notag \\
&& \ \ \ -me^{-\delta s}\int\rho^{p+m-2}\left(\rho-\rho_{M(s)}\right)dy \notag \\
&&	\leq-\int\rho\left\vert\nabla\left(\frac{m}{m-1}\rho^{m-1}+\frac{1}{2}\vert y\vert^{2}\right)\right\vert^{2}dy \notag \\
&& \ \ \ +me^{-\delta s}\Vert\rho\Vert_{\infty}^{p+m-2}\Vert\rho-\rho_{M(s)}\Vert_{1},
\end{eqnarray}
which implies by \eqref{upper bound of the relative entropy 2}, \eqref{CK inequality 3} and Corollary \ref{decay estimate of rho} that
\begin{equation}\label{E E rho M 2}
	\begin{aligned}
		&\frac{d}{ds}E(\rho|\rho_{M(s)})\leq-2E(\rho|\rho_{M(s)})+Ce^{-\delta s}\sqrt{E(\rho|\rho_{M(s)})},
	\end{aligned}
\end{equation}
where $C=C(\Vert u_0 \Vert_{1},\Vert u_0 \Vert_{\infty},n,p,m)$.
If the relative entropy $E(\rho|\rho_{M(s)})$ is strictly positive on any time interval $(s_1,s_2)$, then
\begin{equation}\label{E E rho M 3}
	\begin{aligned}
		&\frac{d}{ds}E(\rho|\rho_{M(s)})^{\frac{1}{2}}\leq-E(\rho|\rho_{M(s)})^{\frac{1}{2}}+Ce^{-\delta s},
	\end{aligned}
\end{equation}
which yields
\begin{equation}\label{E E rho M 4}
	\begin{aligned}
		&\frac{d}{ds}\left(e^{s}E(\rho|\rho_{M(s)})^{\frac{1}{2}}\right)\leq Ce^{(1-\delta )s}
	\end{aligned}
\end{equation}
for $s\in(s_1,s_2)$. By choosing $s_1$ as small as possible, we have either $s_1=0$ or $E(\rho(s_1)|\rho_{M(s_1)})=0$, and then we get the following
exponential convergence for any time $s>0$
\begin{equation}\label{E E rho M 5}
	\begin{aligned}
		&E(\rho|\rho_{M(s)})\leq Ce^{-2\min{\{1,\delta\}}s},\quad \mathrm{for\; } \delta\neq1,
	\end{aligned}
\end{equation}
and
\begin{equation}\label{E E rho M 6}
	\begin{aligned}
		&E(\rho|\rho_{M(s)})\leq C(1+s)^2e^{-2s},\quad \mathrm{for\; } \delta=1.
	\end{aligned}
\end{equation}
By the Csisz$\mathrm{\acute{a}}$r-Kullback inequality \eqref{CK inequality 3}, we have
\begin{equation*}
	\Vert \rho-\rho_{M(s)}\Vert_{1}\leq Ce^{-\min{\{1,\delta\}}s},\quad \mathrm{for\; } \delta\neq1,
\end{equation*}
and
\begin{equation*}
	\Vert \rho-\rho_{M(s)}\Vert_{1}\leq C(1+s)e^{-s},\quad \mathrm{for\; } \delta=1.
\end{equation*}
Going back to the original variables, it follows that
\begin{equation*}
	\Vert u-U_{\widetilde{M}(t)}\Vert_{1}\leq C(1+\lambda t)^{-\frac{1}{\lambda}\min{\{1,\delta\}}},\quad \mathrm{for\; } \delta\neq1,
\end{equation*}
and
\begin{equation*}
	\Vert u-U_{\widetilde{M}(t)}\Vert_{1}\leq C(1+\lambda t)^{-\frac{1}{\lambda}}\ln{(1+\lambda t)},\quad \mathrm{for\; } \delta=1,
\end{equation*}
where $C=C(\Vert u_0 \Vert_{1},\Vert u_0 \Vert_{\infty},n,p,m)$. Moreover, for $1<r<\infty$, we have
\begin{equation*}
	\begin{aligned}
		\Vert u-U_{\widetilde{M}(t)}\Vert_{r}=&e^{-\frac{n(r-1)}{r}s}\left(\int\vert \rho-\rho_{M(s)}\vert^{r}dy\right)^{\frac{1}{r}}\\
		\leq&e^{-\frac{n(r-1)}{r}s}\Vert \rho-\rho_{M(s)}\Vert_{\infty}^{\frac{r-1}{r}}\Vert \rho-\rho_{M(s)}\Vert_{1}^{\frac{1}{r}}\\
		\leq&C (1+\lambda t)^{-\frac{n(r-1)+\min{\{1,\delta\}}}{\lambda r}},\quad \mathrm{for\; } \delta\neq1,
	\end{aligned}
\end{equation*}
and
\begin{equation*}
	\begin{aligned}
		\Vert u-U_{\widetilde{M}(t)}\Vert_{r}\leq C (1+\lambda t)^{-\frac{n(r-1)+1}{\lambda r}}\ln{(1+\lambda t)},\quad \mathrm{for\; } \delta=1,
	\end{aligned}
\end{equation*}
where $C=C(\Vert u_0 \Vert_{1},\Vert u_0 \Vert_{\infty},n,p,m,r)$.

\section*{Acknowledgements} The work was done when C. Xie visited McGill University as a visiting scholar trainee supported by
the Postdoctoral Program of Guangdong Province International Training Program for Outstanding Young Talents. She would like to express her
sincere thanks for the hospitality of McGill University and Guangdong Province, China. The research of S. Fang was supported by the National Natural Science Foundation of China (No.11271141).  The research of M. Mei was supported by NSERC Grant RGPIN-2022-03374.
The research of Y. Qin  was supported by the TianYuan Special Funds of the NNSF of China with contract number 12226403, the NNSF of China with contract No.12171082, the fundamental research funds for the central universities with contract numbers 2232022G-13, 2232023G-13 and by a grant from Science and Technology Commission of Shanghai.




\begin{thebibliography}{99}

\bibitem{expension 1} J.R. Anderson, Local existence and uniqueness of solutions of degenerate parabolic equations, Comm. Partial Differential Equations 16 (1991), 105-143.
\bibitem{expension 2} J.R. Anderson, Necessary and sufficient conditions for the unique solvability of a nonlinear reaction-diffusion model, J. Math. Anal. Appl. 228 (1998), 483-494.
\bibitem{Application fast 1} G.I. Barenblatt, Scaling, self-similarity, and intermediate asymptotics, Cambridge Univ. Press, Cambridge, 1996.
\bibitem{Asymptotic 1} S. Benachour, R.G. Iagar, P. Laurençot, Large time behavior for the fast diffusion equation with critical absorption, J. Differ. Equ. 260 (2016) 8000-8024.
\bibitem{fast 1} A. Blanchet, M. Bonforte, J. Dolbeault, G. Grillo, J.L. V$\mathrm{\acute{a}}$zquez, Asymptotics of the fast diffusion equation via entropy
estimates, Arch. Ration. Mech. Anal. 191 (2009) 347-385.
\bibitem{fast 2} M. Bonforte, J. Dolbeault, G. Grillo, J.L. V$\mathrm{\acute{a}}$zquez, Sharp rates of decay of solutions to the nonlinear fast diffusion
equation via functional inequalities, Proc. Natl. Acad. Sci. USA 107 (2010) 16459-16464.
\bibitem{fast 3} M. Bonforte, J. Dolbeault, B. Nazaret, N. Simonov, Stability in Gagliardo–Nirenberg-Sobolev inequalities, flows,
regularity and the entropy method, Memoirs AMS (2022) (2021) 171, Preprint, https://arxiv.org/abs/2007.03674.

\bibitem{Application fast 3} M. Bonforte, A. Figalli, The Cauchy-Dirichlet problem for the fast diffusion equation on bounded domains, Nonlinear Anal. 239 (2024) 113394.
\bibitem{fast 4} M. Bonforte, G. Grillo, J.L. V$\mathrm{\acute{a}}$zquez, Special fast diffusion with slow asymptotics, entropy method and flow on a
riemannian manifold, Arch. Rational Mech. Anal. 196 (2010) 631-680.
\bibitem{Existence 2} M. Borelli, M. Ughi, The fast diffusion equation with strong absorption: the instantaneous shrinking phenomenon, Rend. Istit. Mat. Univ. Trieste 26 (1994) 109-140.
\bibitem{Nash inequality 1} E.A. Carlen, M. Loss, Sharp constant in Nash's inequality, International Mathematics Research Notices, 7 (1993) 213-215.
\bibitem{Apply 4} J.A. Carrillo, K. Fellner, Long-time asymptotics via entropy methods for diffusion dominated equations, Asymptotic Anal. 42 (1/2) (2005) 29-54.
\bibitem{Apply 3} J.A. Carrillo, A. J$\mathrm{\ddot{u}}$ngel, P.A. Markowich, G. Toscani, A. Unterreiter, Entropy dissipation methods for degenerate parabolic problems and generalized Sobolev inequalities, Monatsh. Math. 133 (2001) 1-82.
\bibitem{Apply 1} J.A. Carrillo, G. Toscani, Asymptotic L1-decay of solutions of the porous medium equation to self-similarity, Indiana Univ. Math. J. 49 (2000) 113-142.
\bibitem{Begin 2} J.A. Carrillo, G. Toscani, Exponential convergence toward equilibrium for homogeneous Fokker-Planck-type equations, Math. Methods Appl. Sci. 21 (1998) 1269-1286.

\bibitem{Apply 2} J.A. Carrillo, J.L. V$\mathrm{\acute{a}}$zquez, Fine asymptotics for fast diffusion equations, Commun. Part. Diff. Eq., 28 (5 \& 6) (2003) 1023-1056.
\bibitem{upper bound of the relative entropy} M. Del Pino, J. Dolbeault, Best constants for Gagliardo-Nirenberg inequalities and applications to nonlinear diffusions, J. Math. Pures Appl. 81 (2002) 847-875.
\bibitem{fast 5} J. Denzler, H. Koch, R.J. McCann, Higher-order time asymptotics of fast diffusion in Euclidean space: A dynamical
systems approach, Mem. Amer. Math. Soc. 234 (2015) 1101.
\bibitem{Extinction 1} R. Ferreira, J.L. V$\mathrm{\acute{a}}$zquez,  Extinction behaviour for fast diffusion equations with absorption, Nonlinear Anal. 43 (8) (2001) 943-985.
\bibitem{expension 6} V.A. Galaktionov, J.L. V$\mathrm{\acute{a}}$zquez, Continuation of blowup solutions of nonlinear heat equations in several space dimensions, Commun. Pur. Appl. Math. 50 (1) (1997) 01-67.
%\bibitem{expension 3} V.A. Galaktionov, J.L. V$\mathrm{\acute{a}}$zquez, The problem of blow-up in nonlinear parabolic equations, Discrete Contin. Dynam. Systems 8 (2) (2002) 399-433.
\bibitem{Generalized Shannon's inequality} M. Kurokiba, T. Ogawa, Finite-time blow-up for solutions to a degenerate drift-diffusion equation for a fast-diffusion case, Nonlinearity 32 (2019) 2073-2093.
\bibitem{Singular solution 1} G. Leoni, A very singular solution for the porous media equation $u_t=\Delta u^{m}-u^p$ when $0<m<1$, J. Differ. Equ. 132 (1996) 353-376.
\bibitem{Continue} R. Huang,  C.H. Jin, M. Mei, J. X. Yin, Existence and stability of traveling waves for degenerate reaction-diffusion equation with time delay, J. Nonlinear. Sci. 28 (2018) 1011-1042.
%\bibitem{expension 5} Y.X. Li, J.C. Wu, Extinction for fast diffusion equations with nonlinear sources, Electron. J. Differ. Eq. 23 (2005), 1-7.
\bibitem{Existence of solution for approximating equation} C.C. Liu, M. Mei, J.A. Yang, Global stability of traveling waves for nonlocal time-delayed degenerate diffusion equation, J. Differ. Equ. 306 (2022) 60-100.
\bibitem{Apply 5} T. Ogawa, Asymptotic stability of a decaying solution to the Keller-Segel system of degenerate type, Differ. Integral Equ. 21 (2008) 1113-1154.
\bibitem{Apply 6} T. Ogawa, T. Suguro, Asymptotic behavior of a solution to the drift-diffusion equation for a fast-diffusion case, J. Differ. Equ. 307 (2022) 114-136.
\bibitem{CK inequality 1} F. Otto, The geometry of dissipative evolution equations: the porous medium equation, Commun. Partial Differ. Equ. 26 (2001) 101-174.
\bibitem{Existence 1} L.A. Peletier, Source-type solutions of the porous media equation with absorption: the fast diffusion case, Nonlinear Anal. Theor. 14 (2) (1990) 107-121.
\bibitem{Asymptotic 2} L.A. Peletier, J.N. Zhao, Large time behavior of solutions of the porous media equation with absorption: the fast diffusion case, Nonlinear Anal. Theor. 17 (10) (1991) 991-1009.
\bibitem{Equicontinuous of solution for approximating equation} E. Sachsp, Continuity of solutions of a singular parabolic equation, Nonlinear Analy. 7 (1983) 387-409.
%\bibitem{expension 4} A.A. Samarskii, V.A. Galaktionov, S.P. Kurdyumov, A.P. Mikhailov, Blow-up in quasilinear parabolic equations, Translated from the Russian by M. Grinfeld, Walter de Gruyter, Berlin, New York, 1995.
\bibitem{cut-off function} Y. Sugiyama, Partial regularity and blow-up asymptotics of weak solutions to degenerate parabolic systems of porous medium type, Manuscripta Math. 147 (3-4) (2015) 311-363.
\bibitem{Contraction of the difference of exponential function} Y. Sugiyama, Y. Yahagi, Extinction, decay and blow-up for Keller-Segel systems of fast diffusion type, J. Differ. Equ., 250 (2011) 3047-3087.
\bibitem{Begin 1} G. Toscani, Kinetic approach to the asymptotic behaviour of the solution to diffusion equations, Rend. Mat. Appl. (7) 16 (1996) 329-346.
\bibitem{Application fast 2} J.L. V$\mathrm{\acute{a}}$zquez, The porous medium equation: Mathematical Theory, Oxford Univ. Press, Oxford, 2007.
\bibitem{fast 6} J.L. V$\mathrm{\acute{a}}$zquez, Smoothing and decay estimates for nonlinear diffusion equations, Oxford Univ. Press, Oxford, 2006.
\bibitem{pme 1} Z.Q. Wu, J.N. Zhao, H.L. Li, J.X. Yin, Nonlinear diffusion equations, World Scientific Publishing Co., Inc., River Edge, NJ, 2001.
\bibitem{pme 2} T.Y. Xu, S.M. Ji, M. Mei, J.X. Yin, Variational approach of critical sharp front speeds in degenerate diffusion model with time delay, Nonlinearity 33 (2020) 4013-4029.

%\bibitem{Application 1} Y.B. Zeldovich, Y.P. Raizer, Physics of shock waves and high-temperature hydrodynamic phenomena, Academic Press, New York, 1966.
%\bibitem{Application 2} J.G. Berryman, C.J. Holland, Stability of the separable solution for fast diffusion, Arch. Ration. Mech. An. 74 (4) (1980) 379-388.

















%\bibitem{mass} F. Feo, Y.H. Huang, B. Volzone, Long-time asymptotics for nonlocal porous medium equation with absorption or convection, Communications in Contemporary Mathematics, 22 (03)  (2020)  1950015.




\end{thebibliography}


\end{document}